\documentclass[11pt,reqno]{amsart}
\usepackage[left=1in,right=1in,top=1in,bottom=1in]{geometry}
\usepackage{enumitem}
\usepackage{amssymb,mathtools}
\usepackage{graphicx}
\usepackage{xcolor}
\usepackage{float}

\usepackage{tikz}
\usetikzlibrary{arrows.meta,decorations.markings,matrix,trees}

\usepackage{hyperref}
\hypersetup{
  colorlinks=true,
  allcolors=blue
}

\def\XXint#1#2#3{{\setbox0=\hbox{$#1{#2#3}{\int}$}
  \vcenter{\hbox{$#2#3$}}\kern-.5\wd0}}

\newcommand{\al}{\alpha}       
    
\newcommand{\lda}{\lambda}
\newcommand{\om}{\Omega}            
\newcommand{\pa}{\partial}
       
\newcommand{\ud}{\mathrm{d}}
\newcommand{\be}{\begin{equation}} 
\newcommand{\ee}{\end{equation}}
\newcommand{\w}{\omega}

\newcommand{\A}{\mathbf{A}}

\newcommand{\B}{\mathbf{B}}

\newcommand{\cE}{\mathcal{E}}

\newcommand{\cF}{\mathcal{F}}

\newcommand{\bI}{\mathbf{I}}

\newcommand{\Z}{\mathbb{Z}}

\newcommand{\MM}{\mathbb{M}}

\newcommand{\m}{\mathbf{m}}

\newcommand{\n}{\mathbf{n}}

\newcommand{\PP}{\mathbf{P}}

\newcommand{\Q}{\mathbf{Q}}  

\newcommand{\R}{\mathbb{R}}

\newcommand{\Ss}{\mathbb{S}}

\newcommand{\vp}{\varphi}

\newcommand{\T}{\mathrm{T}}
\newcommand{\ga}{\gamma}
\newcommand{\Ga}{\Gamma}

\newcommand{\ift}{\infty} 
\newcommand{\wt}{\widetilde}

\newcommand{\f}{\frac}

\newcommand{\ol}{\overline}

\newcommand{\op}{\operatorname}

\newcommand{\na}{\nabla}

\DeclareMathOperator{\dist}{dist}

\DeclareMathOperator{\tr}{tr}

\DeclareMathOperator{\loc}{loc}

\def\<{\langle}\def\>{\rangle}
\def\({\left(}\def\){\right)}
\def\[{\left[}\def\]{\right]}
\numberwithin{equation}{section}
\theoremstyle{plain}
\newtheorem{thm}{Theorem}[section]

\newtheorem{lem}[thm]{Lemma}
\newtheorem{prop}[thm]{Proposition}

\newtheorem{claim}[thm]{Claim}

\theoremstyle{definition}
\newtheorem{defn}[thm]{Definition}

\theoremstyle{remark}
\newtheorem{rem}[thm]{Remark}

\title[Convergence of Landau-de Gennes minimizers in the Lyuksyutov Regime]{Smooth convergence of Landau-de Gennes minimizers \\ in the Lyuksyutov Regime}

\date{}

\author{Haotong Fu}
\address{School of Mathematical Sciences, Peking University, Beijing 100871, China}
\email{547434974@qq.com}

\author{Huaijie Wang}
\address{School of Mathematical Sciences, Peking University, Beijing 100871, China}
\email{huaijie\_wang@163.com}

\author{Wei Wang}
\address{School of Mathematical Sciences, Peking University, Beijing 100871, China}
\email{gjmtamag@gmail.com,\,\,2201110024@stu.pku.edu.cn}

\begin{document}

\begin{abstract}
In this paper, we analyze global minimizers of the Landau–de Gennes functional in $3$-dimensional domains in the Lyuksyutov regime. We prove smooth convergence to an $\Ss^4$-valued limit and obtain an $O(\mu^{-1})$ estimate for the distance to $\Ss^4$ as $\mu \to +\infty$, thereby extending the convergence results of Dipasquale et al. \cite{DMP21}.
\end{abstract}
\maketitle

\section{Introduction}

\subsection{Main results}

Nematic liquid crystals are materials made up of elongated molecules that have long-range alignment while remaining fluid at larger scales. In this phase, molecules align in preferred local directions, but the alignment may vary across space and may break down near defects or interfaces. The alignment remains unchanged when the molecular direction is reversed, leading to non-orientable configurations that may contain topological singularities. At large scales, nematic liquid crystals are typically modeled using continuum models. Among these, the Landau-de Gennes theory \cite{B17, DG93, MN14, V18} has proved effective in describing nematic phases, including the structure and optical properties of defects observed in experiments \cite{K89, L01}.

The Landau-de Gennes theory represents nematic order by a symmetric, traceless second-order tensor. Such tensors form a $5$-dimensional space
$$
\Ss_0:=\left\{Q\in \mathbb{M}_3(\R):Q^{\T}=Q,\,\,\tr Q=0\right\},
$$
where $ \mathbb{M}_3(\R) $ denotes the space of $ 3\times 3 $ real matrices. With respect to an orthonormal basis, $\Ss_0$ is isomorphic to $\R^5$, and the set of tensors satisfying $|Q|=1$ is naturally identified with the $4$-dimensional sphere $\Ss^4$. Three different phases are recognized: isotropic when $|Q|=0$, uniaxial when two eigenvalues coincide, and biaxial when all eigenvalues are distinct.

Let $\om \subset \R^3$ be a simply connected domain occupied by a nematic liquid crystal. We consider the Landau-de Gennes functional
\be
\cF_{\op{LdG}}(Q):=\int_\om\left( \f{L}{2}|\nabla Q|^2+f_b(Q)\right)\ud x,\label{LdG1}
\ee
where the configuration satisfies $Q\in H^1(\om,\Ss_0)$, and the bulk energy density, $f_b$, is given by
\be
f_b(Q)=-\f{a^2}{2}\tr(Q^2)-\f{b^2}{3}\tr(Q^3)+\f{c^2}{4}\left(\tr(Q^2)\right)^2+C.
\ee
Here $a, b$, and $c$ are material-dependent constants and $C=C(a,b,c)$ is chosen so that $\min_{\Q\in\Ss_0}f_b(Q)=0$. The minima of $f_b$ form a vacuum manifold of uniaxial states, where
$$
\mathcal N:=\left\{s_*\(\mathbf n\otimes\mathbf n-\f{1}{3}\mathbf I\); \mathbf n\in\mathbb S^2\right\}\subset \Ss_0,
$$
where
$$
s_*:=\f{1}{4c^2}(b^2+\sqrt{b^4+24a^2c^2}).
$$
The vacuum manifold $\mathcal N$ is diffeomorphic to the real projective plane $\R\mathbb{P}^2$. Consequently, $\pi_2(\mathcal N)=\mathbb Z$ and $\pi_1(\mathcal N)=\mathbb Z_2$, reflecting the topology responsible for defect formation. The Landau-de Gennes framework is often studied in regimes where the elastic constant $L$ is small. This situation arises when the macroscopic length scale of the domain is much larger than the material's intrinsic microscopic length scale. In such regimes, one can set $L=\varepsilon^2\ll 1$ and rescale the energy to obtain the energy density $\f{1}{2}|\nabla Q|^2+\f{1}{\varepsilon^2}f_b(Q)$. Formally setting $\varepsilon= 0$ reduces the Landau-de Gennes functional \eqref{LdG1} to the Dirichlet energy for maps taking values in $\mathcal N$. 

A substantial body of work has addressed the asymptotics of minimizers of \eqref{LdG1} concerning the vanishing elasticity limit (i.e., the elastic parameter $\varepsilon\to 0^+$). In $3$-dimensional domains with uniformly bounded total energy, it was shown in \cite{MZ10} that global minimizers of \eqref{LdG1} converge to uniaxial configurations governed by the one-constant Oseen-Frank theory, with strong convergence away from finitely many point defects. The analysis was subsequently generalized in \cite{Can17}, where the energy is allowed to have a logarithmic divergence as $\varepsilon \to 0^+$. In this setting, the defect set includes both line defects and discrete point singularities. For further analysis of defects, we refer to works like \cite{FWW25a, FWW25b, Geng23}. 

Another direction concerns the low-temperature limit, which is extensively studied in \cite{Can15, CL17, HMO17}. A further limiting setting is the Lyuksyutov regime, in which the limiting map takes values in $\Ss^4$. To describe this regime, we rewrite the Landau–de Gennes energy in the form used in \cite{DMP21}. We rescale the tensor field by setting
$$
\Q(x):=\sqrt{\f{3}{2}}\f{1}{s_*}Q(x).
$$
Then the energy functional becomes
$$
\cF_{\op{LdG}}(Q)=\f{2}{3}s_*^2L\cF_{\lambda,\mu}(\Q),
$$
where
\be
\cF_{\lda,\mu}(\Q):=\int_{\om}\(\f{1}{2}|\na\Q|^2+f_{\lda,\mu}(\Q)\)\ud x.\label{LdG}
\ee
Here the reduced parameters $\lambda$ and $\mu$ are given by
$$
\lambda:=\sqrt{\f{2}{3}}\f{b^2s_*}{L},\quad \mu:=\f{a^2}{L}.
$$
Formally, the low-temperature limit and the small elastic constant limit correspond to the regimes $ a^2\to+\ift $, $ b^2=|a|^{-1} $, and $ L\to 0^+ $, respectively. The reduced energy density $f_{\lambda,\mu}$ becomes
\be
f_{\lda,\mu}(\Q):=\lda W(\Q)+\f{\mu}{4}(1-|\Q|^2)^2,
\ee
where
\be
W(\Q):=\f{1}{4\sqrt{6}}|\Q|^4-\f{1}{3}\tr\Q^3+\f{1}{12\sqrt{6}}.
\ee
In particular, when $|\Q|=1$, i.e. $\Q\in\Ss^4$, we have
\be
W(\Q)=\f{1}{3\sqrt{6}}\left(1-\sqrt{6}\f{\tr(\Q^3)}{|\Q|^3}\right),
\ee
which penalizes only biaxiality. The constraint $|\Q|=1$ is called the \emph{Lyuksyutov constraint} in the series of works by Dipasquale et al. \cite{DMP21, DMP24a, DMP24b}. The limiting problem as $\mu \to +\infty$ is called the Lyuksyutov regime. It was proved in \cite{DMP21} that, as $\mu\to+\infty$ with $\lambda$ fixed, there exists a minimizer $\Q_\lambda$ of
$$
\cE_{\lda}(\Q):=\int_{\om}\(\f{1}{2}|\na\Q|^2+\lda W(\Q)\)\ud x,\quad \Q\in H^1(\om,\Ss^4),
$$
such that the global minimizers of \eqref{LdG} converge strongly to $\Q_\lambda$ in $H^1$. In this paper, we extend this result by proving that such convergence can be extended to the $C^\infty$-topology. Under the above renormalization, the analysis in \cite{MZ10} corresponds to the regime $\lambda\to \infty$ and $\mu\to \infty$ with $\lambda\sim \mu$. For defect analysis in this direction, we refer to \cite{GW25, TY23, Yu20}.

We have the corresponding Euler-Lagrange equation
\be
\begin{aligned}
\Delta\Q&=D_{\Q}f_{\lda,\mu}(\Q)+\f{\lda}{\sqrt{6}}|\Q|^2\bI=\lda\(\f{1}{\sqrt{6}}|\Q|^2\Q+\f{1}{3}|\Q|^2\bI-\Q^2\)+\mu(|\Q|^2-1)\Q,
\end{aligned}\label{EulerLagrange}
\ee
where $ D_{\Q}f_{\lda,\mu}(\Q)\in\mathbb{M}_3(\R) $ is defined by
$$
(D_{\Q}f_{\lda,\mu}(\Q))_{ij}=\f{\pa f_{\lda,\mu}}{\pa\Q_{ij}}(\Q)\quad\text{for any }i,j\in\{1,2,3\}.
$$

Define
$$
e_{\lda,\mu}(\Q):=\f{1}{2}|\na\Q|^2+f_{\lda,\mu}(\Q).
$$
We now state the main theorem of this paper.

\begin{thm}\label{main1}
Let $\lambda,\mu>0$, let $\om$ be a bounded $C^3$ domain, and let $\Q_b\in C^2(\pa\om,\Ss^4)$. Assume that $\{\Q_{\lambda,\mu}\}_{\mu>0}$ is a family of minimizers of
$$
\min\{\cF_{\lda,\mu}(\Q):\Q\in H^1(\om,\Ss_0),\,\,\Q=\Q_b\text{ on }\pa\om\}.
$$
Then there exists $\Q_{\lambda}\in H^1(\om,\Ss^4)$ such that the following properties hold.
\begin{enumerate}[label=$(\theenumi)$]
\item\label{pr1} For each $\lambda>0$, $\Q_{\lambda}$ is a minimizer of
$$
\min\{\cE_{\lda}(\Q):\Q\in H^1(\om,\Ss^4),\,\,\Q=\Q_b\text{ on }\pa\om\},
$$
and as $\mu\to+\infty$
$$
\Q_{\lambda,\mu}\to\Q_{\lambda}\text{ strongly in }H^1(\om,\Ss_0).
$$
\item\label{pr2}There exists a constant $C>0$, depending only on $\om$ and $\Q_b$, such that
$$
\limsup_{\mu\to+\ift}\mu(1-|\Q_{\lda,\mu}|)\leq C.
$$
\item\label{pr3} For any $ \al\in(0,1) $, up to a subsequence, 
$$
\Q_{\lda,\mu}\to\Q_{\lda}\text{ in }C_{\loc}^{\ift}(\om,\Ss_0)\cap C^{1,\al}(\ol{\om},\Ss_0).
$$
\end{enumerate}
\end{thm}

\begin{rem}\label{boundedenergy}
Since $\Ss^4$ is simply connected, $\Q_b\in C^2(\pa\om,\Ss^4)$, and $\om$ is a bounded $C^3$ domain, it follows from the finite-energy extension result \cite[Theorem 6.2]{HL87} that there exists $\PP\in H^1(\om,\Ss^4)$ such that
$$
\int_{\om}|\na\PP|^2\leq C.
$$
In particular, 
$$
\cF_{\lda,\mu}(\PP)\leq C.
$$
The existence of minimizers $\{\Q_{\lambda,\mu}\}_{\mu>0}$ then follows from the direct method of the calculus of variations. Moreover, the minimizing property of $\Q_{\lambda,\mu}$ implies
$$
\cF_{\lda,\mu}(\Q_{\lda,\mu})\leq\cF_{\lda,\mu}(\PP)\leq C,
$$
where $C>0$ is independent of $\mu$. Hence, under the assumptions of Theorem \ref{main1}, the values $\cF_{\lambda,\mu}(\Q_{\lambda,\mu})$ are uniformly bounded with respect to $\mu$.
\end{rem}

\begin{rem}\label{Qldasmooth}
The Theorem \ref{boundedenergy}\ref{pr1} follows from \cite[Theorem 1.3]{DMP21}. The main difference between the Lyuksyutov regime and the small-elasticity limit studied in \cite{MZ10} is that the limiting map $\Q_{\lambda}$ is smooth (see \cite[Theorem 1.2]{DMP21}). Consequently, no defect analysis is required, and the convergence results in Theorem \ref{boundedenergy}\ref{pr2}$\&$\ref{pr3} are sharp and global.
\end{rem}

\subsection{Main idea and the sketch of the proof} We follow the arguments developed in \cite{MZ10} and \cite{NZ13}.

We first establish interior and boundary partial regularity estimates. For the interior case, we show that if $\Q\in C^{\infty}(B_{2r},\Ss_0)$ solves \eqref{EulerLagrange} and satisfies
\be
\f{1}{r}\int_{B_{2r}}e_{\lda,\mu}(\Q)\ll 1\quad\Longrightarrow\quad \sup_{B_r}e_{\lda,\mu}(\Q)\leq C,\label{partialregularityuse}
\ee
where $C>0$ is independent of $\mu$. As summarized in \cite{CL22}, the argument relies on the following three ingredients.
\begin{itemize}
\item The uniform boundedness 
$$
\|\Q_{\lda,\mu}\|_{L^{\ift}(B_{2r})}\leq C,
$$
where $ C>0 $ does not depend on $ \mu $.
\item A monotonicity formula
$$
\f{\ud}{\ud\rho}\(\f{1}{\rho}\int_{B_{\rho}(x)}e_{\lda,\mu}(\Q)\)\geq 0
$$
valid whenever $ x\in B_{2r} $.
\item The Bochner-type inequality
\be
-\Delta e_{\lda,\mu}(\Q)\leq Ce_{\lda,\mu}^2(\Q)\label{Bochner1}
\ee
with a constant $C>0$ independent of $\mu$.
\end{itemize}
The first two properties are straightforward in this model. However, the Bochner-type inequality \eqref{Bochner1} is difficult to obtain. Instead, we prove
\be
-\Delta e_{\lda,\mu}(\Q)\leq C(e_{\lda,\mu}^2(\Q)+\lda^2).\label{Bochner2}
\ee
Compared with \eqref{Bochner1}, inequality \eqref{Bochner2} contains the additional term $C\lambda^2$. This term is acceptable and allows us to adapt the arguments of \cite{MZ10}. Indeed, after the standard scaling 
$$
\PP(y):=\Q(x_0+ry),
$$
the map $\PP$ satisfies \eqref{EulerLagrange} with $\lambda$ replaced by $r\lambda$ and $\mu$ replaced by $r\mu$. Inequality \eqref{Bochner2} then becomes
$$
-\Delta e_{r\lda,r\mu}(\PP)\leq C(e_{r\lda,r\mu}(\PP)+\lda^2r^2).
$$
As $r\to0^+$, the error term becomes small. This observation allows us to adapt the methods of \cite{MZ10} to the present setting.

As noted in Remark \ref{Qldasmooth}, the limiting map $\Q_{\lambda}$ is smooth. Therefore the small-energy condition \eqref{partialregularityuse} holds for sufficiently small radii $r>0$. We then follow the argument of \cite{NZ13} to obtain uniform regularity of $\Q_{\lambda,\mu}$ with respect to $\mu$ and to prove higher-order convergence.

\subsection{Notations} We use the following conventions in this paper.

\begin{itemize}
\item Throughout the paper, $C$ denotes a positive constant. To indicate dependence on parameters $a_1,a_2,\dots$, we write $C(a_1,a_2,\dots)$. The value of $C$ may change from line to line.
\item We use the Einstein summation convention, summing over repeated indices.
\item $ \MM_3(\R) $ denotes the space of $3\times3$ real matrices, and $\bI$ denotes the identity matrix.
\item Matrices are denoted by bold capital letters such as $\A$ and $\B$. Unless otherwise stated, matrices belong to $ \MM_3(\R) $.
\item For $ \A,\B\in\MM_3(\R) $,
$$
\A: \B:=\A_{ij}\B_{ij},\quad |\A|^2=\A:\A.
$$
\item Let $ \A,\B:\om\subset\R^3\to\mathbb{M}_3(\R) $ be differentiable matrix-valued functions. We write
$$
\na\A=D\A:=(\pa_1\A,\pa_2\A,\pa_3\A),
$$
and
$$
\na\A:\na \B=D\A:D\B:=\pa_{\ell}\A_{ij}\pa_{\ell}\B_{ij},\quad |\na\A|^2=|D\A|^2=\na\A:\na\A.
$$
For $ j\in\Z_+ $, 
$$
|D^j\A|^2:=\sum_{|\al|=j}|\pa^{\al}\A|^2,
$$
where $ \al $ is the multi-index.
\item For $\m,\n\in\R^3$,
$$
\m\cdot\n=\n_i\m_i,\quad\m\otimes\n\in\mathbb{M}_3(\R),\,\,(\m\otimes\n)_{ij}=(\m_i\n_j).
$$
\item For $ x\in\R^3 $ and $ r>0 $, we denote
$$
B_r(x):=\{x\in\R^3:|y-x|<r\}.
$$
When $x=0$, we write $B_r$.
\end{itemize}

Next, we recall the regularity assumption of a bounded domain.

\begin{defn}\label{DefnLocal}
Let $U\subset\R^3$ be a bounded domain. We say that $U$ is a $C^k$ domain, where $k\in\Z_+$, if there exist constants
\be
r_{U,k}>0,\quad M_{U,k}>0,
\label{r0M0}
\ee
and a continuous nondecreasing function
\be
\omega_k:[0,+\infty)\to[0,+\infty),\quad
\lim_{t\to0^+}\omega_k(t)=0,\label{wkdef}
\ee
such that the following holds. For every $x_0\in\pa U$, there exist a $C^k$ function $\psi:\R^2\to\R$, a translation and rotation of coordinates such that $x_0=(0,0,0)$ and
\be
U_r(x_0):=U\cap B_r(x_0)\{(y_1,y_2,y_3):y_3>\psi(y_1,y_2)\}\cap B_r(x_0)\label{UcapB}
\ee
for all $ 0<r<20(M_{U,k}+1)r_{U,k} $, and
\be
\begin{gathered}
\psi(0,0)=0,\quad\|D^i\psi\|_{L^{\ift}(\R^2)}\leq M_{U,k}\quad\text{for }i=1,2,...,k,\\
|\psi(x')-\psi(y')|\leq\w_k(|x'-y'|)\quad\text{for any }x',y'\in\R^2.
\end{gathered}\label{Dkest}
\ee
We call $U$ a bounded $C^k$ domain with parameters $ M_{U,k},r_{U,k} $, and $ \w_k $ if $ U $ is a bounded domain satisfying \eqref{UcapB} and \eqref{Dkest}, where $M_{U,k},r_{U,k},\omega_k$ if \eqref{UcapB}-\eqref{Dkest} hold. When $x_0=0$, we write $U_r:=U_r(x_0)$. We also define
\begin{align*}
\ol{U}_r(x_0):=\ol{U}\cap\ol{B}_r(x)=\{(y_1,y_2,y_3)\in\R^3:(y_1,y_2)\in\R^2,\,\,y_3\geq\psi(y_1,y_2)\}\cap\ol{B}_r(x_0),\\
\Ga_r^U(x_0):=\pa U\cap B_r(x)=\{(y_1,y_2,y_3)\in\R^3:(y_1,y_2)\in\R^2,\,\,y_3=\psi(y_1,y_2)\}\cap B_r(x_0).
\end{align*}
\end{defn}

\section{Preliminaries}

\subsection{Monotonicity} We recall the monotonicity formula corresponding to \eqref{LdG}.

\begin{prop}\label{Mono}
Let $ U\subset\R^3 $ be a bounded domain. Assume that $ \Q\in C^{\infty}(U,\Ss_0) $ satisfies \eqref{EulerLagrange}. For $ x_0\in\overline{U} $, define
$$
\Theta_{\lambda,\mu}(\Q;x_0,r):=\frac{1}{r}\int_{U_r(x_0)} e_{\lambda,\mu}(\Q)\,dx.
$$
Then the following properties hold.

\begin{enumerate}[label=$(\theenumi)$]

\item If $ x_0\in U $, then
$$
\Theta_{\lambda,\mu}(\Q;x_0,r)
\le
\Theta_{\lambda,\mu}(\Q;x_0,R)
$$
for any $ 0<r<R<\dist(x_0,\partial U) $.

\item If $ U $ is a $ C^3 $ domain with parameters $ M_{U,3},r_{U,3},\omega_3 $, $ x_0\in\partial U $, and $ \Q=\Q_b\in C^2(\partial U,\Ss^4) $, then for any $
0<r<R<r_{U,3} $, we have
$$
\Theta_{\lambda,\mu}(\Q;x_0,r)\le\Theta_{\lambda,\mu}(\Q;x_0,R)+C(R-r),
$$
where $ C>0 $ depends only on $ M_{U,3},r_{U,3},\omega_3,\|\nabla\Q\|_{L^2(U)} $, and $ \Q_b $.
\end{enumerate}
\end{prop}
\begin{proof}
This is exactly \cite[Lemma 4.4]{DMP21}. See also
\cite[Proposition 2.4]{DMP21} and
\cite[Lemmas 2 and 9]{MZ10}.
\end{proof}

\subsection{Bochner-type inequality} In this subsection, we establish a Bochner-type inequality with an error term for solutions of \eqref{EulerLagrange}.

\begin{lem}\label{Bochner}
Assume $ \Q\in C^{\ift}(\om,\Ss_0) $ satisfies \eqref{EulerLagrange} with $ |\Q|\leq 1 $ in $ \om $. There exists an absolute constant $ C>0 $ such that if $ |\Q|\in[\f{1}{2},1] $ in $ \om $, then
$$
-\Delta e_{\lda,\mu}(\Q)\leq C(e_{\lda,\mu}^2(\Q)+\lda^2).
$$
\end{lem}
\begin{proof}
Since $ \tr\Q=0 $, a direct calculation and the Cauchy-Schwarz inequality give
\be
\begin{aligned}
-\Delta e_{\lda,\mu}(\Q)&=-\na(\Delta\Q):\na\Q-|D^2\Q|^2-\pa_k\(D_{\Q}f_{\lda,\mu}(\Q):\pa_k\Q\)\\
&\quad\quad\stackrel{\eqref{EulerLagrange}}{\leq}-2\pa_k(D_{\Q}f_{\lda,\mu}(\Q)):\pa_k\Q-D_{\Q}f_{\lda,\mu}(\Q):\Delta\Q\\
&\quad\quad=2\lda\pa_k(\Q^2):\pa_k\Q-\f{4(\lda+\sqrt{6}\mu)}{\sqrt{6}}|\na\Q:\Q|^2\\
&\quad\quad\quad\quad+\mu(1-|\Q|^2)|\na\Q|^2-|\Delta\Q|^2-\f{2\lda}{\sqrt{6}}|\Q|^2|\na\Q|^2\\
&\quad\quad\leq C\lda^2+C|\na\Q|^4+\f{\mu^2}{16}(1-|\Q|^2)^2-|\Delta\Q|^2.
\end{aligned}\label{DeltanablaQ}
\ee
Since $ \|\Q\|_{L^\infty(\Omega)}\le 1 $, another application of the Cauchy-Schwarz inequality yields
\begin{align*}
|\Delta\Q|^2&\geq\mu^2(1-|\Q|^2)^2|\Q|^2+2\lda\mu\(\f{1}{\sqrt{6}}|\Q|^2\Q+\f{1}{3}|\Q|^2\bI-\Q^2\):[(|\Q|^2-1)\Q]\\
&=\mu^2(1-|\Q|^2)^2|\Q|^2+\f{2\lda\mu}{\sqrt{6}}(|\Q|^2-1)(|\Q|^4-\tr\Q^3)\\
&\ge\f{\mu^2}{4}(1-|\Q|^2)^2-C\lda^2(|\Q|^4-\tr\Q^3)^2-\f{\mu^2}{8}(1-|\Q|^2)^2\\
&\geq \f{\mu^2}{8}(1-|\Q|^2)^2-C\lda^2.
\end{align*}
This, together with \eqref{DeltanablaQ}, implies that
\begin{align*}
-\Delta e_{\lda,\mu}(\Q)\leq C\lda^2+C|\na\Q|^4-\f{\mu^2}{16}(1-|\Q|^2)^2\leq C(e_{\lda,\mu}^2(\Q)+\lda^2),
\end{align*}
completing the proof.
\end{proof}

\subsection{Estimates on the boundary} To obtain boundary regularity (see Section \ref{BoundaryPartialReg}), we need estimates for solutions of \eqref{EulerLagrange} near the boundary.

\begin{prop}\label{boundary1}
Let $ \lambda,\mu>0 $ and let $ U $ be a bounded $ C^3 $ domain with parameters $ M_{U,3},r_{U,3},\omega_3 $. Assume that $ \Q\in C^{\ift}(U_r(x_0),\Ss_0)\cap C^2(\ol{U}_r(x_0),\Ss_0) $ solves \eqref{EulerLagrange} and 
$$
\Q\in\Ss^4\text{ on }\Ga_r^U(x_0) 
$$
for some $ x_0\in\pa U $ and $ 0<r<r_{U,3} $. If $ |\Q|\in[\f{1}{2},1] $ in $ U_r(x_0) $, then $ \Q^*=\f{\Q}{|\Q|} $ is well defined and
\be
\begin{aligned}
&\|\na(\Q-\Q^*)\|_{L^{\ift}(\Ga_{\f{r}{2}}^U(x_0))}\\
&\quad\quad\leq C\(r^{-1}\|(1-|\Q|)\|_{L^{\ift}(U_r(x_0))}+r^{\f{1}{4}}\|\na\Q\|_{L^2(U_r(x_0))}^{\f{1}{2}}\|\na\Q\|_{L^{\ift}(U_r(x_0))}^{\f{3}{2}}\)\\
&\quad\quad\quad\quad+C\(r^{-\f{1}{2}}\|\na\Q\|_{L^2(U_r(x_0))}\|\na\Q\|_{L^{\ift}(U_r(x_0))}+\lda r\),
\end{aligned}\label{naQQstar}
\ee
where $ C>0 $ depends only on $ M_{U,3},r_{U,3} $, and $ \w_3 $.
\end{prop}

We first recall some basic estimates for elliptic equations,  beginning with the boundary version of Caccioppoli's inequality.

\begin{lem}\label{Caccioppoli}
Let $U$ be a bounded $C^1$ domain with parameters $M_{U,1}, r_{U,1}$, and $\omega_1$. Assume that $\Delta u = 0$ in $U_{2r}(x_0)$ and $u=0$ on $\partial U \cap B_{2r}(x_0)$, where $x_0\in \partial U$ and $0<r<\f{1}{2}r_{U,1} $. Then
\be
\|\nabla u\|_{L^2(B_r(x_0))} \le Cr^{-1}\|u\|_{L^2(B_{2r}(x_0))}.
\label{Caccio}
\ee
\end{lem}
\begin{proof}
Let $\varphi\in C_0^\infty(B_{2r}(x_0),[0,1])$ satisfy $\varphi\equiv1$ in $B_r(x_0)$ and $ \|\na\vp\|_{L^{\ift}(B_{2r}(x_0))}\leq Cr^{-1} $ for some absolute constant $ C>0 $. Since $u=0$ on $\Gamma_{2r}^U(x_0)$, we test the equation $\Delta u=0$ with $u\varphi^2$ and obtain
$$
\int_{U_{2r}(x_0)}\vp^2|\na u|^2\leq\left|\int_{U_{2r}(x_0)}2u\vp(\na u\cdot\na\vp)\right|.
$$
The estimate \eqref{Caccio} follows from the Cauchy-Schwarz inequality.
\end{proof}

The next lemma gives localized Lipschitz estimates for Poisson's equation.

\begin{lem}\label{LempRe}
Let $U$ be a bounded $C^3$ domain with parameters $M_{U,3}, r_{U,3}$, and $\omega_3$. Assume that $x_0\in\partial U$, $p>3$, and $0<r<r_{U,3}$. Let $ F\in L^p(U_{2r}(x_0)) $, $ g\in C^2(\Ga_{2r}^U(x_0)) $, and $ u\in W^{2,p}(U_{2r}(x_0)) $ solve
$$
\left\{\begin{aligned}
\Delta u&=F&\text{ in }& U_{2r}(x_0),\\
u&=g&\text{ on }&\Ga_{2r}^U(x_0).
\end{aligned}\right.
$$
We have the estimate
\begin{align*}
\|\na u\|_{L^{\ift}(U_r(x_0))}&\leq C\|D_{\pa U}g\|_{L^{\ift}(\Ga_{2r}^U(x_0))}+Cr\|D_{\pa U}^2g\|_{L^{\ift}(\Ga_{2r}^U(x_0))}\\
&\quad\quad+Cr^{-\f{3}{2}}\|\na u\|_{L^2(U_r(x_0))}+Cr^{1-\f{3}{p}}\|F\|_{L^p(U_{2r}(x_0))},
\end{align*}
where $D_{\partial U}^j$ $(j\in\mathbb Z_+)$ denotes tangential derivatives on $\partial U$, and $C>0$ depends only on $p,M_{U,3},r_{U,3},\omega_3$.
\end{lem}
\begin{proof}
This lemma is a special case of \cite[Lemma 11]{NZ13}.
\end{proof}

\begin{proof}[Proof of Proposition \ref{boundary1}]
After translation, we assume $x_0=0$. Let 
$$
u:=1-|\Q|^2\geq 0.
$$
Using \eqref{EulerLagrange}, a direct computation shows that in $U_r$
\be
-\Delta u=2\lda\(\f{1}{\sqrt{6}}|\Q|^4-\tr\Q^3\)-2\mu |\Q|^2u+2|\na\Q|^2\leq C_0(|\na\Q|^2+\lda),\label{Deltau}
\ee
for some absolute constant $C_0>0$. Let $w$ and $w_1$ solve the Dirichlet problems
\be
\begin{aligned}
-\Delta w&=C_0(|\na\Q|^2+\lda)\quad\text{in }U_r\quad\text{and}\quad w=u\text{ on }\pa(U_r),\\
-\Delta w_1&=C_0(|\na\Q|^2+\lda)\quad\text{in }U_r\quad\text{and}\quad w_1=0\text{ on }\pa(U_r).
\end{aligned}\label{eqveqw1}
\ee
Define $ w_2:=w-w_1 $. Then
\be
-\Delta w_2=0\quad\text{in }U_r\quad\text{and}\quad w_2=u\quad\text{on }\pa(U_r).\label{eqw2}
\ee
Since $\Q\in\Ss^4$ on $\Gamma_r^U$, we have $u=0$ on $\Gamma_r^U$, and therefore $w=0$ on $\Gamma_r^U$. Applying Lemma \ref{LempRe} with $p=4$ and using the fundamental theorem of calculus gives
\be
\begin{aligned}
w(y)&\leq C\|\na w\|_{L^{\ift}(U_{\f{2r}{3}})}\dist(y,\pa U)\\
&\leq C\(r^{-\f{3}{2}}\|\na w\|_{L^2(U_{\f{3r}{4}})}+r^{\f{1}{4}}\|\na\Q\|_{L^2(U_{\f{3r}{4}})}^{\f{1}{2}}\|\na\Q\|_{L^{\ift}(U_{\f{3r}{4}})}^{\f{3}{2}}+\lda r\)\dist(y,\pa U).
\end{aligned}\label{vyestimate}
\ee
Moreover, the $ L^2 $-estimate of $ w_1 $ implies
\be
\begin{aligned}
\|\na w_1\|_{L^2(U_r)}&\leq Cr\|(|\na\Q|^2+\lda)\|_{L^2(U_r)}\leq Cr(\|\na\Q\|_{L^2(U_r)}\|\na\Q\|_{L^{\ift}(U_r)}+\lda r^{\f{3}{2}}).\label{w1estimate}
\end{aligned}
\ee
From \eqref{eqw2}, the maximum principle yields
$$
\|w_2\|_{L^{\ift}(U_r)}\leq \|u\|_{L^{\ift}(U_r)}.
$$
Using Lemma \ref{Caccioppoli} and the assumption $|\Q|\in[\tfrac12,1]$ in $U_r$ gives
\be
\begin{aligned}
\|\na w_2\|_{L^2(U_{\f{3r}{4}})}&\leq Cr^{-1}\|w_2\|_{L^2(U_r)}\leq Cr^{\f{1}{2}}\|u\|_{L^{\ift}(U_r)}\\
&\leq Cr^{\f{1}{2}}\|(1-|\Q|^2)\|_{L^{\ift}(U_r)}\leq Cr^{\f{1}{2}}\|(1-|\Q|)\|_{L^{\ift}(U_r)}.
\end{aligned}\label{w2estimate}
\ee
It follows from \eqref{Deltau} and \eqref{eqveqw1} that
$$
-\Delta(u-w)\leq 0\quad\text{in }U_r,\quad\text{and}\quad u-w=0\quad\text{on }\pa(U_r).
$$
Testing this inequality with $\max\{0,u-w\}$ yields
$$
u(y)\le w(y)\quad\text{for any }y\in U_r.
$$
Since $1-|\Q|\le u$, we obtain
$$
1-|\Q(y)|\le w(y)\quad\text{for any }y\in U_r.
$$
Combining \eqref{vyestimate}, \eqref{w1estimate}, and \eqref{w2estimate} gives, for any $y\in U_r$,
\begin{align*}
|(\Q-\Q^*)(y)|&\leq Cr^{-1}\|(1-|\Q|)\|_{L^{\ift}(U_r)}\dist(y,\pa U)\\
&\quad\quad+Cr^{\f{1}{4}}\|\na\Q\|_{L^2(U_{\f{3r}{4}})}^{\f{1}{2}}\|\na\Q\|_{L^{\ift}(U_{\f{3r}{4}})}^{\f{3}{2}}\dist(y,\pa U)\\
&\quad\quad+C\(r^{-\f{1}{2}}\|\na\Q\|_{L^2(U_r)}\|\na\Q\|_{L^{\ift}(U_r)}+\lda r\)\dist(y,\pa U).
\end{align*}
Since $\Q\in\Ss^4$ on $\Gamma_r^U$, this implies \eqref{naQQstar}.
\end{proof}

\section{Interior regularity}

\subsection{Partial regularity} In this section, we focus on interior partial regularity results.

\begin{prop}\label{InteriorRegularity}
Let $ r\in(0,1] $ and $ x\in\R^3 $. Assume $ \Q\in C^{\ift}(B_{2r}(x),\Ss_0) $ solves \eqref{EulerLagrange}. There exists an absolute constant $ \delta\in(0,1) $ such that if $ |\Q|\in[\f{1}{2},1] $ in $ B_{2r}(x) $ and $ \Theta_{\lda,\mu}(\Q;x,r)<\delta $, then
$$
r^2\|e_{\lda,\mu}(\Q)\|_{L^{\ift}(B_r(x))}\leq C,
$$
where $ C>0 $ depends only on $ \lda $.
\end{prop}

An essential tool for the proof is the following Harnack-type inequality:

\begin{lem}\label{Harnack}
Let $ c_0,r>0 $, and $ f\in L^{\ift}(B_r) $. Assume that $ u\in H^1(B_r) $ is a subsolution of the equation $ -\Delta u+c_0u=f $, i.e., for any $ 0\leq\vp\in C_0^1(B_r) $,
$$
\int_{B_r}(\na u\cdot\na\vp+c_0u\vp)\leq\int_{B_r}f\vp.
$$
Then the positive part $ u^+:=\max\{u,0\}\in L^{\ift}(B_r) $, and in particular
$$
\sup_{B_{\f{r}{2}}} u^+\leq C\(\f{1}{r^3}\int_{B_r}u^++r\|f\|_{L^{\ift}(B_r)}^2\),
$$
where $ C>0 $ depends only on $ c_0 $.
\end{lem}
\begin{proof}
By scaling, we may assume $ r=1 $. The result follows from \cite[Theorem 4.1]{HL11}.
\end{proof}

\begin{proof}[Proof of Proposition \ref{InteriorRegularity}]
After a translation, we set $ x=0 $. Choose $ 0<r_1<2r $ to satisfy
\be
\sup_{0\leq\rho\leq 2r}\left\{(2r-\rho)^2\sup_{B_{\rho}}e_{\lda,\mu}(\Q)\right\}=(2r-r_1)^2\sup_{B_{r_1}}e_{\lda,\mu}(\Q).\label{defr1}
\ee
Let 
$$
\alpha_\mu := \sup_{B_{r_1}} e_{\lambda,\mu}(\mathbf{Q})= e_{\lambda,\mu}(\mathbf{Q})(x_1)
$$
for some $x_1 \in \overline{B}_{r_1}$. Let $C_0 = C_0(\lambda) > 4$ be a constant to be determined. If $\alpha_\mu (2r-r_1)^2 \leq C_0$, then \eqref{defr1} implies 
$$
r^2 \sup_{B_r} e_{\lambda,\mu}(\mathbf{Q}) \leq C_0,
$$ 
which completes the proof. Consequently, we assume
\be
r_{\mu}:=\al_{\mu}^{\f{1}{2}}(2r-r_1)>C_0^{\f{1}{2}}>2.\label{alldamugeq4}
\ee
For any $ |y|<1 $, we have
\be
|x_1+\al_{\mu}^{-\f{1}{2}}y|\leq r_1+\al_{\mu}^{-\f{1}{2}}\leq r+\f{r_1}{2}\leq 2r.\label{x1alldamu}
\ee
This allows us to define $ \mathbf{P}(y) := \mathbf{Q}(x_1 + \alpha_\mu^{-\f{1}{2}}y)$ in $B_1$. Since $\mathbf{Q}$ solves \eqref{EulerLagrange}, the rescaled tensor $\mathbf{P}$ satisfies
$$
\Delta\PP=\wt{\lda}\(\f{1}{\sqrt{6}}|\PP|^2\PP+\f{1}{3}|\PP|^2\bI-\PP^2\)+\wt{\mu}(|\PP|^2-1)\PP,
$$
where $ \wt{\lda}:=\al_{\mu}^{-1}\lda $ and $ \wt{\mu}:=\al_{\mu}^{-1}\mu $. Define 
$$
w_\mu(y) := e_{\wt{\lambda},\wt{\mu}}(\mathbf{P})(y).
$$
By Lemma \ref{Bochner}, we have
\be
-\Delta w_{\mu}\leq C( w_{\mu}^2+\al_{\mu}^{-2}\lda^2).\label{Deltawtw}
\ee
It follows from \eqref{defr1}-\eqref{x1alldamu}, and the definition of $ \al_{\mu} $ that $ w_{\mu}(0)=1 $ and
\be
\sup_{B_1}w_{\mu}\leq\al_{\mu}^{-1}\sup_{B_{\f{2r+r_1}{2}}}e_{\lda,\mu}(\Q)\leq\al_{\mu}^{-1} (4\al_{\mu})\leq 4.\label{bound4}
\ee
Thus, \eqref{Deltawtw} reduces to 
$$
-\Delta w_\mu \leq C(w_\mu + \alpha_\mu^{-2} \lambda^2).
$$
Applying Lemma \ref{Harnack}, we obtain
$$
1=w_{\mu}(0)\leq C_1\(\int_{B_1} w_{\mu}+\f{\lda^4}{\al_{\mu}^4}\),
$$
where $ C_1>0 $ is an absolute constant. If $C_1 \alpha_\mu^{-2} \lambda^2 > \f{1}{2}$, then $\alpha_\mu$ is bounded by a constant depending on $\lambda$, which contradicts \eqref{alldamugeq4} for large $C_0$. Thus, we assume $C_1 \alpha_\mu^{-2} \lambda^2 \leq \f{1}{2}$, yielding
$$
\f{1}{2}\leq C_1\int_{B_1} w_{\mu}.
$$
Scaling back to $\mathbf{Q}$ and invoking the monotonicity formula (Proposition \ref{Mono}), we find
$$
\f{1}{2}\leq \f{C}{r_{\mu}}\int_{B_{\f{r_{\mu}}{2}}} w_{\mu}\leq C\Theta_{\lda,\mu}(\Q;x_1,r)\leq C_2\Theta_{\lda,\mu}(\Q;0,2r),
$$
where $ C_2>0 $ is also an absolute constant. By choosing $\delta$ such that $C_2 \delta <\f{1}{4}$, we reach a contradiction. This concludes the proof.
\end{proof}

\subsection{Higher order regularity} This subsection addresses interior higher-order regularity, assuming the boundedness of the energy density $e_{\lambda,\mu}(\mathbf{Q})$. We follow the approach established in \cite[Section 6]{NZ13}.

\begin{prop}\label{Ckestimate}
Let $\lambda, \mu > 0$, $0 < r \leq 1$, and $x_0 \in \mathbb{R}^3$. Let $\mathbf{Q} \in C^\infty(B_r(x_0), \mathcal{S}_0)$ be a solution of \eqref{EulerLagrange} satisfying $|\mathbf{Q}| \leq 1$ and
\be
\|(\mu(1-|\Q|^2)+|\na\Q|)\|_{L^{\ift}(B_{r}(x_0))}\leq C_0.\label{0deriva}
\ee
Then for any $j \in \mathbb{Z}_{\geq 0}$ and $0 < \rho < r$, we have
\be
\|(|D^j(\mu(1-|\Q|^2))|+|D^{j+1}\Q|)\|_{L^{\ift}(B_{\rho})}\leq C(1+(r-\rho)^{-j}),\label{regimprove}
\ee
where $ C>0 $ is a constant depending only on $ C_0,j $ and $ \lda $.
\end{prop}

To prove the above proposition, we need the following auxiliary lemmas.

\begin{lem}[\cite{BBH93}, Lemma A.1]\label{Intp}
Let $U \subset \mathbb{R}^3$ be a bounded domain. If $u \in C^2(U)$ and $f \in L^\infty(U)$ satisfy $-\Delta u = f$, then for any $x \in U$,
$$
|\na u(x)|^2\leq C\(\|f\|_{L^{\ift}(U)}\|u\|_{L^{\ift}(U)}+\dist^{-2}(x,\pa U)\|u\|_{L^{\ift}(U)}^2\),
$$
where $ C>0 $ is an absolute constant.
\end{lem}

\begin{lem}\label{corau}
Let $\gamma \geq 0$ and $r > 0$. Suppose $u \in W^{1,\infty}(B_r) \cap C^0(\overline{B}_r)$ and $F \in L^\infty(B_r)$ satisfy $-L\Delta u + c_0 u = F$ in the weak sense, with $L, c_0 > 0$. Then for any $x \in B_r$,
\begin{align*}
|u(x)|&\leq\f{1}{c_0}\|F\|_{L^{\ift}(B_r)}+C\min\left\{1,\f{L^{\ga}}{c_0^{\ga}(r-|x|)^{2\ga}}\right\}\|u\|_{L^{\ift}(\pa B_r)},
\end{align*}
where $ C>0 $ depends only on $ \ga $.
\end{lem}
\begin{proof}
It follows from \cite[Lemma 6]{NZ13} and the maximum principle.
\end{proof}

\begin{proof}[Proof of Proposition \ref{Ckestimate}]
Let $ u_{\mu}:=\mu(1-|\Q|^2) $. For $ 0<s\le r $ and $ j\in\Z_{\ge 0} $, define
$$
\begin{aligned}
d_{\mu}(s,j)&:=\left\|\(\sum_{\ell=0}^j|D^{\ell}u_{\mu}|+\sum_{\ell=0}^{j+1}|D^{\ell}\Q|\)\right\|_{L^{\ift}(B_s)},\\
D_{\mu}(s,j)&:=d_{\mu}(s,j)+\sum_{\substack{\ell_1+\cdots+\ell_p=j+1\\ p\ge 2,\;\ell_q\ge1}}\prod_{k=1}^pd_{\mu}(s,\ell_k),\\
\Phi_{\mu}(s,j)&:=\|(|D^{j+1}u_{\mu}|+|D^{j+2}\Q|)\|_{L^{\ift}(B_{s})}.
\end{aligned}
$$
From \eqref{EulerLagrange} and \eqref{Deltau} we obtain
\begin{align}
\Delta\Q&=\lda\(\f{1}{\sqrt{6}}|\Q|^2\Q+\f{1}{3}|\Q|^2\bI-\Q^2\)-u_{\mu}\Q,\label{DeltaQeq}\\
\mu^{-1}\Delta u_{\mu}&=2\lda\(\tr\Q^3-\f{1}{\sqrt{6}}|\Q|^4\)+2|\Q|^2 u_{\mu}-2|\na\Q|^2.\label{Deltaumu}
\end{align}
Applying $D^{j+1}$ to \eqref{DeltaQeq} yields, for any $0<t\le r$,
$$
\|\Delta(D^{j+1}\Q)\|_{L^{\ift}(B_t)}\leq C(\|D^{j+1}u_{\mu}\|_{L^{\ift}(B_{t})}+D_{\mu}(t,j)).
$$
By Lemma \ref{Intp}, for any $0<\delta<1$ and $0<s<t\le r$,
\be
\begin{aligned}
\|D^{j+2}\Q\|_{L^{\ift}(B_s)}&\leq C(s-t)^{-1}\|D^{j+1}\Q\|_{L^{\ift}(B_t)}+C\|\Delta(D^{j+1}\Q)\|_{L^{\ift}(B_t)}^{\f{1}{2}}\|D^{j+1}\Q\|_{L^{\ift}(B_t)}^{\f{1}{2}}\\
&\leq C(\delta^{-1}+(t-s)^{-1})D_{\mu}(t,j)+\delta\|D^{j+1}u_{\mu}\|_{L^{\ift}(B_{t})}\\
&\leq C(\delta^{-1}+(t-s)^{-1})D_{\mu}(t,j)+\delta\Phi_{\mu}(t,j),
\end{aligned}\label{QE1}
\ee
where we also use the Cauchy-Schwarz inequality. Taking $ D^j $ for both sides of \eqref{Deltaumu} gives
$$
\|\Delta(D^ju_{\mu})\|_{L^{\ift}(B_{t})}\leq C\mu(\|D^ju_{\mu}\|_{L^{\ift}(B_{t})}+D_{\mu}(t,j))
$$
for any $ 0<t\leq r $. By applying Lemma \ref{Intp} again, for $ 0<s<t\leq r $,
\be
\begin{aligned}
\|D^{j+1}u_{\mu}\|_{L^{\ift}(B_s)}&\leq C(s-t)^{-1}\|D^ju_{\mu}\|_{L^{\ift}(B_t)}+C\|\Delta(D^ju_{\mu})\|_{L^{\ift}(B_t)}^{\f{1}{2}}\|D^ju_{\mu}\|_{L^{\ift}(B_t)}^{\f{1}{2}}\\
&\leq C(\mu^{\f{1}{2}}+(s-t)^{-1})D_{\mu}(t,j).
\end{aligned}\label{umupre}
\ee
We rewrite \eqref{Deltaumu} as
$$
-\mu^{-1}\Delta u_{\mu}+2u_{\mu}=2\lda\(\f{1}{\sqrt{6}}|\Q|^4-\tr\Q^3\)+2\mu^{-1}u_{\mu}^2+2|\na\Q|^2:=g_{\mu}.
$$
Using \eqref{0deriva}, \eqref{QE1}, and \eqref{umupre}, we obtain
$$
\|D^{j+1}g_{\mu}\|_{L^{\ift}(B_{\f{s+t}{2}})}\leq C(\delta^{-1}+(t-s)^{-1})D_{\mu}(t,j)+\delta \Phi_{\mu}(t,j)
$$
for any $ 0<\delta<1 $. Since
$$
-\mu^{-1}\Delta(D^{j+1}u_{\mu})+2D^{j+1}u_{\mu}=D^{j+1}g_{\mu},
$$
Lemma \ref{corau} together with \eqref{umupre}, applied with
$L=\mu^{-1}$ and $\gamma=\frac{1}{2}$, gives
\be
\begin{aligned}
\|D^{j+1}u_{\mu}\|_{L^{\ift}(B_{s})}&\leq C\|D^{j+1}g_{\mu}\|_{L^{\ift}(B_{\f{s+t}{2}})}+C\min\{1,\mu^{-\f{1}{2}}(t-s)^{-1}\}\|D^{j+1}u_{\mu}\|_{L^{\ift}(B_{\f{s+t}{2}})}\\
&\leq\delta\Phi_{\mu}(t,j)+C(\delta^{-1}+(t-s)^{-1})D_{\mu}(t,j).
\end{aligned}\label{trXz}
\ee
This, together with \eqref{QE1}, implies that for any $ 0<\delta<1 $ and $ 0<s<t\leq r $,
\be
\Phi_{\mu}(s,j)\leq\delta\Phi_{\mu}(t,j)+C(\delta^{-1}+(t-s)^{-1})D_{\mu}(t,j).\label{Phisj}
\ee
Choose $ \delta=\f{1}{2} $ and define
\begin{align*}
\rho_0=s,\quad\text{and}\quad\rho_{i+1}-\rho_i=\f{2^i}{3^{i+1}}(t-s).
\end{align*}
It follows from \eqref{Phisj} that
$$
\Phi_{\mu}(s,j)\leq\f{1}{2^i}\Phi_{\mu}(\rho_i,j)+C(1+(t-s)^{-1})D_{\mu}(t,j)\cdot\sum_{\ell=0}^{i-1}\(\f{3}{4}\)^{\ell}.
$$
Letting $ i\to+\ift $ gives
$$
\Phi_{\mu}(s,j)\leq C(1+(t-s)^{-1})D_{\mu}(t,j).
$$
Interpolation and \eqref{0deriva} imply
$$
d_{\mu}(s,j+1)\leq C(1+(t-s)^{-1})D_{\mu}(t,j).
$$
Therefore \eqref{regimprove} follows by induction on $j$.
\end{proof}

\section{Boundary regularity}\label{BoundaryPartialReg}

\subsection{Main properties} We study boundary regularity for solutions of \eqref{EulerLagrange}. Similar results appear in \cite[Section 5]{NZ13}, but our model requires more refined estimates.

\begin{prop}\label{smallreg2}
Let $\lambda,\mu>0$. Let $U$ be a bounded $C^3$ domain with parameters $M_{U,3}$, $r_{U,3}$, and $\omega_3$. There exists $0<\delta<1$, depending only on $M_{U,3}, r_{U,3}$, and $\omega_3$, such that the following holds. Suppose $ \Q\in C^{\ift}(U_{2r}(x_0),\Ss_0)\cap C^2(\ol{U}_{2r}(x_0),\Ss_0) $ solves \eqref{EulerLagrange} and satisfies $ \Q=\Q_b\in C^2(\Ga_{2r}^U(x_0),\Ss^4) $ on $\Ga_{2r}^U(x_0)$ for some $0<r<\frac12 r_{U,3}$. Assume
\be
\begin{gathered}
|\Q|\in(1-\delta,1]\text{ in }U_{2r}(x_0),\\
E_{\mu}:=\sup_{x\in U_{2r}(x_0)}\sup_{0<\rho<2r}\Theta_{\lda,\mu}(\Q;x,\rho)<\delta.
\end{gathered}\label{Qestimateass}
\ee
Then
\be
\|e_{\lda,\mu}(\Q)\|_{L^{\ift}(U_{\f{r}{2}}(x_0))}\leq C(b_{\mu}^2+r^{-2}),\label{eldamuprove}
\ee
where 
$$
b_{\mu}:=\|(|\na_{\pa U}\Q_b|+r|D_{\pa U}^2\Q_b|)\|_{L^{\ift}(\Ga_{2r}^U(x_0))},
$$
and $ C>0 $ depends only on $ \lda,M_{U,3},r_{U,3},\w_3 $.
\end{prop}

The following proposition improves regularity near the boundary.

\begin{prop}\label{LiftQNll1}
Let $\lambda,\mu>0$ and let $U\subset\mathbb R^3$ be a bounded $C^1$ domain with parameters $ M_{U,1},r_{U,1} $, and $ \w_1 $. Let $x_0\in\partial U$ and $0<r<\frac{1}{2} r_{U,1}$. Suppose $ \Q\in C^{\ift}(U_r(x_0),\Ss_0)\cap C^0(\ol{U}_r(x_0),\Ss_0) $ solves \eqref{EulerLagrange} and satisfies $ \Q\in\Ss^4 $ on $ \Ga_{2r}^U(x_0) $. Assume $ |\Q|\in[\f{1}{2},1] $ in $ U_{2r}(x_0) $. Then
$$
\|\Delta\Q\|_{L^{\ift}(B_r(x_0))}\leq C(\|\na\Q\|_{L^{\ift}(U_{2r}(x_0))}+\lda+r^{-2}),
$$
where $ C>0 $ depends only on $ \lda,M_{U,1},r_{U,1} $, and $ \w_1 $.
\end{prop}

\subsection{Proof of Proposition \ref{smallreg2}}

After a translation, we assume $x_0=0$. We prove that if $\delta\in(0,1)$ is sufficiently small and
\eqref{Qestimateass} holds, then
\be
M_{\mu}:=\sup_{0\leq\rho\leq r}\left\{(r-\rho)^2\sup_{U_\rho}(e_{\lda,\mu}(\Q)-C_0b_{\mu}^2)\right\}\leq C_0,\label{Mvaboun}
\ee
where $ C_0=C_0(\lda,M_{U,3},r_{U,3},\w_3)>0 $ will be chosen later. Once \eqref{Mvaboun} holds, \eqref{eldamuprove} follows immediately. Choose $0<r_1<r$ and $x_1\in\overline U_{r_1}$ such that
\begin{gather*}
M_{\mu}=(r-r_1)^2\sup_{U_{r_1}}(e_{\lda,\mu}(\Q)-C_0b_{\mu}^2),\\
e_{\lda,\mu}(\Q)(x_1)=\sup_{U_{r_1}}e_{\lda,\mu}(\Q):=\al_{\mu}.
\end{gather*}
Note that
\be
M_{\mu}=4r_2^2(\al_{\mu}-C_0b_{\mu}^2),\label{Mvavalu}
\ee
where $ r_2:=\f{r-r_1}{2} $. A direct computation gives
\be
\begin{aligned}
\sup_{U_{r_2}(x_1)}e_{\lda,\mu}(\Q)&\leq\sup_{U_{r_2+r_1}}e_{\lda,\mu}(\Q)\leq \sup_{U_{r_2+r_1}}(e_{\lda,\mu}(\Q)-C_0b_{\mu}^2)+C_0b_{\mu}^2\\
&=\f{(r-(r_2+r_1))^2\(\sup\limits_{U_{r_2+r_1}}e_{\lda,\mu}(\Q)-C_0b_{\mu}^2\)}{(r-(r_2+r_1))^2}+C_0b_{\mu}^2\\
&\leq\f{4r_2^2(\al_{\mu}-C_0b_{\mu}^2)}{r_2^2} +C_0b_{\mu}^2\leq 4\al_{\mu},\\
\end{aligned}\label{inftbounwva}
\ee
Define
$$
U':=\{y\in\R^3:x_1+\al_{\mu}^{-\f{1}{2}}y\in U\}\quad\text{and}\quad\rho_1:=\al_{\mu}^{\f{1}{2}}r_2.
$$
For $ y\in U_{\rho_1}':=U'\cap B_{\rho_1} $, let
$$
\PP(y):=\Q(x_1+\al_{\mu}^{-\f{1}{2}}y)\quad\text{and}\quad w_{\mu}(y):=e_{\wt{\lda},\wt{\mu}}(\PP)(y),
$$
where $ \wt{\lda}:=\al_{\mu}^{-1}\lda $ and $ \wt{\mu}:=\al_{\mu}^{-1}\mu $. Similar to \eqref{bound4}, we deduce from \eqref{inftbounwva} that
\be
\sup_{U_{\rho_1}'}w_{\mu}\leq 4.\label{wva4boun}
\ee
By Lemma \ref{Bochner}, we have
\be
-\Delta w_{\mu}\leq C(w_{\mu}+\al_{\mu}^{-2}\lda^2)\quad\text{in }U_{\rho_1}'.\label{HarBa}
\ee
Scaling back to $ e_{\lda,\mu}(\Q) $, it follows from \eqref{Qestimateass} that
\be
\sup_{0<\rho\leq\rho_1}\(\f{1}{\rho}\int_{U_{\rho}'}w_{\mu}\)<\delta.\label{Boundenwva}
\ee
Define
$$
\rho_2:=\min\left\{\dist(0,\pa U'),\f{\rho_1}{10}\right\}.
$$

\begin{claim}\label{claim1}
If $\delta>0$ is sufficiently small, then either $\rho_2 \le 1$ or \eqref{eldamuprove} holds.
\end{claim}
\begin{proof}
Assume $\rho_2>1$. Then $\rho_1>10$ and $\dist(0,\partial U')>1$, which imply that $B_1\subset U'$. By the definition of $\alpha_\mu$, we have $w_\mu(0)=1$. Applying \eqref{Boundenwva} and Lemma \ref{Harnack}, we obtain
\be
1\leq C\(\int_{B_1}w_{\mu}+\f{\lda^4}{\al_{\mu}^4}\)\leq C'\(\delta+\f{\lda^4}{\al_{\mu}^4}\).\label{CintB1}
\ee
Here $C'>0$ is an absolute constant. If $ \f{C'\lda^4}{\al_{\mu}^4}>\f{1}{2} $, then $ \al_{\mu}\leq (2C')^{\f{1}{4}}\lda $. Since $ 0<r_2<\f{r}{2}<\f{r_{U,3}}{2} $, \eqref{Mvavalu} gives
$$
M_{\mu}=4r_2^2(\al_{\mu}-C_0b_{\mu}^2)\leq (2C')^{\f{1}{4}}\lda r_{U,3}^2,
$$
which implies \eqref{Mvaboun} and hence \eqref{eldamuprove}. On the other hand, if $ \f{C'\lda^4}{\al_{\mu}^4}<\f{1}{2} $, then \eqref{CintB1} yields $ \f{1}{2}\leq C'\delta $. Choosing $\delta>0$ such that $C'\delta<\frac14$ gives a contradiction.
\end{proof}

With Claim \ref{claim1}, we may assume without loss of generality that $\rho_2\le 1$. If $ \rho_2=\f{\rho_1}{10} $, then $ \rho_1^2=\al_{\mu}r_2^2\leq 100 $, which again implies \eqref{Mvaboun}. Hence we assume
\be
\rho_2=\dist(0,\pa U')<\f{\rho_1}{10}.\label{rho210rho1}
\ee

\begin{claim}\label{claim2}
Let $ \PP^*=\f{\PP}{|\PP|} $ and $ \Q^*=\f{\Q}{|\Q|} $. For any $ 0<\rho<\f{\rho_1}{10} $,
\be
\begin{aligned}
\|\na(\PP-\PP^*)\|_{L^{\ift}(\Ga_{\rho}^{U'})}&\leq C\rho^{-1}\|(1-|\PP|)\|_{L^{\ift}(U_{5\rho}')}\\
&\quad\quad+C\rho^{-\f{1}{2}}\|\na\PP\|_{L^2(U_{5\rho}')}\|\na\PP\|_{L^{\ift}(U_{5\rho}')}\\
&\quad\quad+C\rho^{\f{1}{4}}\|\na\PP\|_{L^2(U_{5\rho}')}^{\f{1}{2}}\|\na\PP\|_{L^{\ift}(U_{5\rho}')}^{\f{3}{2}}+C\al_{\mu}^{-1}\lda\rho,
\end{aligned}\label{Boundaryuse2}
\ee
and
\be
\begin{aligned}
\|\na\PP^*\|_{L^{\ift}(\Ga_{\rho}^{U'})}&\leq C\al_{\mu}^{-\f{1}{2}}b_{\mu}+C\rho^{-\f{3}{2}}\|\na\PP\|_{L^2(U_{5\rho}')}+C\al_{\mu}^{-1}\lda\rho\\
&\quad\quad+C\rho^{\f{1}{4}}\|\na\PP\|_{L^2(U_{5\rho}')}^{\f{1}{2}}\|\na\PP\|_{L^{\ift}(U_{5\rho}')}^{\f{3}{2}},
\end{aligned}\label{Boundaryuse21}
\ee
where $ C>0 $ depends only on $ M_{U,3},r_{U,3} $, and $ \w_3 $.
\end{claim}
\begin{proof}
Fix $\delta\in(0,1/2)$ and set $ s:=\al_{\mu}^{-\f{1}{2}}\rho $. Since $\rho\in(0,\f{\rho_1}{10})$ and $\rho_1=\alpha_\mu^{\f{1}{2}}r_2$, we have $ s\in(0,\f{r_2}{10}) $. From \eqref{rho210rho1} we obtain
$$
\dist(x_1,\pa U)=\al_{\mu}^{-\f{1}{2}}\rho_2<\f{r_2}{10}\leq\f{r}{10}<\f{r_{U,3}}{10}.
$$
If $ 0<s<\dist(x_1,\pa U) $, then $ \pa U\cap B_s=\pa U'\cap B_{\rho}=\emptyset $ and the claim is trivial. Assume $s\ge\dist(x_1,\partial U)$. Choose $x_2\in\partial U$ such that $ |x_1-x_2|=\dist(x_1,\pa U) $. Then
$$
B_s(x_1)\subset B_{2s}(x_2)\quad\text{and}\quad B_{4s}(x_2)\subset B_{5s}(x_1).
$$
Using Lemma \ref{boundary1}, there is $ C=C(M_{U,3},r_{U,3},\w_3)>0 $ such that
\begin{align*}
&\|\na(\Q-\Q^*)\|_{L^{\ift}(\Ga_s^U(x_1))}\leq \|\na(\Q-\Q^*)\|_{L^{\ift}(\Ga_{2s}^U(x_2))}\\
&\quad\quad\leq Cs^{-1}\|(1-|\Q|)\|_{L^{\ift}(U_{5s}(x_1))}+Cs^{-\f{1}{2}}\|\na\Q\|_{L^2(U_{5s}(x_1))}\|\na\Q\|_{L^{\ift}(U_{5s}(x_1))}\\
&\quad\quad\quad\quad+Cs^{\f{1}{4}}\|\na\Q\|_{L^2(U_{5s}(x_1))}^{\f{1}{2}}\|\na\Q\|_{L^{\ift}(U_{5s}(x_1))}^{\f{3}{2}}+C\lda s.
\end{align*}
Changing variables from $\Q$ to $\PP$ yields \eqref{Boundaryuse2}.

A direct computation gives
\be
\Delta\Q^*=\f{\mathbb{P}_{\Q^*}(\Delta\Q)}{|\Q|}+\f{\Q}{|\Q|^3}(|\na|\Q||^2-|\na\Q|^2)+\f{2}{|\Q|^3}(\Q:\pa_k\Q)\pa_k\Q,\label{DeltaQstares}
\ee
where $ \mathbb{P}_{\Q^*} $ is the projection of points in $ \Ss_0 $ to the tangent plane of $ \Ss^{4} $ at $ \Q^* $. It follows from \eqref{EulerLagrange} that
$$
\mathbb{P}_{\Q^*}(\Delta\Q)=\lda\mathbb{P}_{\Q^*}\(\f{1}{3}|\Q|^2\bI-\Q^2\).
$$
Using \eqref{DeltanablaQ} and $|\Q|\in[\f{1}{2},1]$ gives
$$
|\Delta\Q^*|\leq C(|\na\Q|^2+\lda).
$$
This, together with Lemma \ref{LempRe} with $ p=4 $, and the property that $ \Q=\Q^* $ on $ \Ga_{2r}^U $, implies that
\begin{align*}
\|\na\Q^*\|_{L^{\ift}(\Ga_s^U(x_1))}&\leq C\|\na_{\pa U}\Q\|_{L^{\ift}(\Ga_{5s}^U(x_1))}+Cs\|D_{\pa U}^2\Q\|_{L^{\ift}(\Ga_{5s}^U(x_1))}\\
&\quad\quad+Cs^{-\f{3}{2}}\|\na\Q^*\|_{L^2(U_{5s}(x_1))}+Cs^{\f{1}{4}}\|\Delta\Q^*\|_{L^4(U_{5s}(x_1))}\\
&\leq Cb_{\mu}+Cs^{-\f{3}{2}}\|\na\Q\|_{L^2(U_{5s}(x_1))}+C\lda s\\
&\quad\quad+Cs^{\f{1}{4}}\|\na\Q\|_{L^2(U_{5s}(x_1))}^{\f{1}{2}}\|\na\Q\|_{L^{\ift}(U_{5s}(x_1))}^{\f{3}{2}},
\end{align*}
where $ C=C(M_{U,3},r_{U,3},\w_3)>0 $. Scaling from $\Q$ to $\PP$ yields \eqref{Boundaryuse21}.
\end{proof}

Using \eqref{wva4boun}, \eqref{Boundenwva}, \eqref{Boundaryuse2}, and \eqref{Boundaryuse21}, we have
\be
\|\na\PP\|_{L^{\ift}(\Ga_{\rho}^{U'})}\leq C(\rho^{-2}(1-|\Q|)^2+\delta+\rho\delta^{\f{1}{2}}+\al_{\mu}^{-2}\lda^2\rho^2+\al_{\mu}^{-1}b_{\mu}^2+\rho^{-2}\delta),\label{nablaQboundary}
\ee
where $ C_1=C_1(M_{U,3},r_{U,3},\w_3)>0 $. We assume $ \al_{\mu}\geq C_0b_{\mu}^2 $, since otherwise \eqref{Mvaboun} is trivial. Since $\Q\in\mathbb S^4$ on $\partial U\cap B_{2r}$, we have $
|f_{\wt{\lda},\wt{\mu}}(\Q)|\leq C\al_{\mu}^{-1}\lda $. Combining this with \eqref{nablaQboundary} gives
$$
\sup_{\Ga_{\rho}^U}w_{\mu}\leq C_1(C_0^{-1}+\al_{\mu}^{-2}\lda^2\rho^2+\al_{\mu}^{-1}\lda+\rho^{-2}\delta+\delta+\rho\delta^{\f{1}{2}}):=C_1C(\rho)
$$
for any $ 0<\rho<\f{\rho_1}{10} $, where we have used $ |\Q|\in(1-\delta,1] $ and $ \delta\in(0,1) $. Define
$$
\wt{w}_{\mu}=\left\{\begin{aligned}
&\max\left\{C_1C(\rho),w_{\mu}\right\}&\text{ in }&U_{\rho}',\\
&C_1C(\rho)&\text{ in }&B_{\rho}\backslash U'.
\end{aligned}\right.
$$
We have $ \wt{w}_{\mu}\in W^{1,\ift}(B_{\rho}) $. By \eqref{HarBa}, it satisfies
$$
-\Delta \wt{w}_{\mu}\leq C(\wt{w}_{\mu}+\al_{\mu}^{-2}\lda^2)
$$
in $ B_{\rho} $. Recall $ \rho_1=\al_{\mu}^{\f{1}{2}}r_2\leq\al_{\mu}^{\f{1}{2}}r_{U,3} $. We further assume
\be
\al_{\mu}\geq\max\left\{\(\f{1}{5}C_2\lda^4r_{U,3}\)^{\f{2}{7}},\f{4}{25}C_2\lda^2r_{U,3}^2,16C_2\lda\right\},\label{almuchoose}
\ee
since otherwise \eqref{Mvaboun} holds with an appropriately chosen
$C_0=C_0(\lambda,M_{U,3},r_{U,3},\omega_3)>0$. Condition \eqref{almuchoose} implies
$$
\f{1}{2}\leq C_2C(\rho),\quad \f{1}{100}\al_{\mu}^{-2}\lda^2\rho_1^2<\f{1}{16},\text{ and }\al_{\mu}^{-1}\lda<\f{1}{16}.
$$
Hence
$$
C_2\wt{C}(\rho):=C_2(C_0^{-1}+\rho^{-2}\delta+\delta+\rho\delta^{\f{1}{2}})\geq\f{3}{8}
$$
Applying \eqref{Boundenwva}, Lemma \ref{Harnack}, and using $w_\mu(0)=1$, we obtain
$$
1\leq C\(\f{1}{\rho^3}\int_{B_{\rho}}w_{\mu}+\f{\lda^4\rho}{\al_{\mu}^4}\)\leq C_2(C(\rho)+\al_{\mu}^{-4}\lda^4\rho)
$$
for any $0<\rho<\f{\rho_1}{10}$. The second inequality uses the regularity of $\partial U'$ after scaling. Choose $\delta>0$ and $C_0>0$ such that
$$
C_0>16C_2,\quad0<\delta<\f{1}{16},\text{ and }\delta^{\f{1}{3}}+\delta^{\f{5}{6}}<\f{1}{16},
$$
Let $\rho_3=\delta^{1/3}$. Then $ C_2C(\rho_3)<\f{3}{16} $, which implies that $ \f{\rho_1}{10}<\rho_3 $. Therefore
$$
\f{3}{8}\leq C_2\wt{C}\(\f{\rho_1}{20}\)\leq \f{1}{16}+\f{400 C_2}{\rho_1^2}+\f{1}{16}+\f{1}{32}.
$$
Hence 
$$
M_{\mu}\leq 4\rho_1^2\leq C(\lda,M_{U,3},r_{U,3},\w_3),
$$
which proves \eqref{Mvaboun}.

\subsection{Proof of Proposition \ref{LiftQNll1}} Compared with Lemma \ref{corau}, we have the following boundary version. We omit the proof for simplicity.

\begin{lem}\label{boundaryconpa}
Let $ \ga\geq 0 $ and $ U $ be a bounded $ C^1 $ domain with parameters $ M_{U,1},r_{U,1} $, and $ \w_1 $. Assume that $ u\in W^{1,\ift}(U_r(x_0))\cap C^0(U_r(x_0)) $ with $ x_0\in\pa U $ satisfies in the weak sense $ -L\Delta u+c_0u=F $ in $ U_r(x_0) $ and $ u=0 $ on $ \Ga_r^U(x_0) $ for some $ 0<r<r_{U,1} $, where $ L,c_0>0 $ and $ F\in L^{\ift}(U_r(x_0)) $. Then for any $ x\in U_r(x_0) $,
$$
|u(x)|\leq\f{1}{c_0}\|F\|_{L^{\ift}(U_r(x_0))}+C\min\left\{1,\f{L^{\ga}}{c_0^{\ga}(r-|x|)^{2\ga}}\right\}\|u\|_{L^{\ift}(U_r(x_0))},
$$
where $ C>0 $ depends only on $ \ga $.
\end{lem}

\begin{proof}[Proof of Proposition \ref{LiftQNll1}]
Let $ u:=(1-|\Q|^2)^2 $. Using \eqref{EulerLagrange} and direct computation, we obtain
\begin{align*}
-\Delta u&=-2\mu|\Q|^2u-2|\na(|\Q|^2)|^2+4|\na\Q|^2(1-|\Q|^2)\\
&\quad\quad+2\lda\(\f{1}{\sqrt{6}}|\Q|^4-\tr(\Q^3)\)(1-|\Q|^2).
\end{align*}
Using the Cauchy–Schwarz inequality and the assumption $ |\Q|\geq \frac{1}{2} $ in $ U_{2r}(x_0) $, we obtain
\begin{align*}
\mu^{-1}\Delta u&\geq 2|\Q|^2u-4\mu^{-1}(|\na\Q|^2+C\lda)(1-|\Q|^2)\\
&\geq\f{1}{4}u-C(\lda\mu^{-1}+\mu^{-2}(\|\na\Q\|_{L^{\ift}(U_{2r}(x_0))}^4+\lda^2)).
\end{align*}
Since $ \Q\in\Ss^4 $ on $ \pa U\cap B_{2r}(x_0) $, we have $ u=0 $ on $ \Ga_{2r}^U(x_0) $. Applying Lemma \ref{boundaryconpa} with $ L=\mu^{-1} $, $ c_0=\f{1}{2} $, and $ \ga=2 $, we obtain
$$
\|u\|_{L^{\ift}(U_r(x_0))}\leq C\mu^{-2}(\|\na\Q\|_{L^{\ift}(U_{2r}(x_0))}^4+\lda^2+r^{-4}).
$$
Consequently, \eqref{EulerLagrange} implies
\begin{align*}
\|\Delta\Q\|_{L^{\ift}(U_r(x_0))}\leq C\lda+C\mu\|\sqrt{u}\|_{L^{\ift}(U_{2r}(x_0))}\leq C(\|\na\Q\|_{L^{\ift}(U_{2r}(x_0))}^2+\lda+r^{-2}),
\end{align*}
which completes the proof.
\end{proof}
\section{Proof of Theorem \ref{main1}}

As in Remark \ref{Qldasmooth}, we only prove the second and third properties. By Remark \ref{boundedenergy} and \cite[Theorem 1.2]{DMP21}, we have 
$$
\|\Q\|_{L^{\ift}(\om)}\leq 1\quad\text{and}\quad\cF_{\lda,\mu}(\Q_{\lda,\mu})\leq C(\Q_b,\om).
$$
Therefore, the monotonicity formula in Proposition \ref{Mono} applies to $ \Q_{\lda,\mu} $. By \cite[Theorem 1.2]{DMP21}, the limiting map $ \Q_{\lda} $ is smooth for each fixed $ \lda>0 $. Hence there exists $ r_0>0 $ such that for all $ 0<r<2r_0 $,
$$
\f{1}{r}\int_{\om\cap B_r(x)}\(\f{1}{2}|\na\Q_{\lda}|^2+\lda W(\Q_{\lda})\)<\delta,
$$
where $ \delta\in(0,1) $ will be chosen later. Up to a subsequence, $ \Q_{\lda,\mu}\to\Q_{\lda} $ strongly in $ H^1(\om,\Ss_0) $. Then for sufficiently large $ \mu>0 $
$$
\f{1}{r_0}\int_{\om\cap B_{r_0}(x)}e_{\lda,\mu}(\Q_{\lda,\mu})<2\delta.
$$
for any $ x\in\ol{\om} $. By Propositions \ref{InteriorRegularity} and \ref{smallreg2}, we can choose $ \delta>0 $ sufficiently small so that
$$
\|e_{\lda,\mu}(\Q_{\lda,\mu})\|_{L^{\ift}(\om)}\leq C(\om,\Q_b).
$$
This implies the second property of Theorem \ref{main1}. Additionally, for sufficiently large $ \mu>0 $, we have $ |\Q|\geq\f{1}{2} $ in $ \om $. Now, the $ C_{\loc}^{\ift} $ convergence follows from Proposition \ref{Ckestimate} and the $ C^{1,\al} $ $ (\al\in(0,1)) $ convergence is the consequence of Proposition \ref{LiftQNll1}.

\section*{Acknowledgments} 

The authors are grateful to Professor Adriano Pisante for helpful discussions that inspired the initial idea of this work. This work was partially supported by the National Key R$\&$D Program of China under Grant 2023YFA1008801 and NSF of China under Grant 12288101.

\section*{Declarations} 

\subsection*{Data availability} This article has no associated data.
\subsection*{Conflict of interest} The authors declare that they have no conflict of interest.

\bibliographystyle{plain}
\bibliography{ConvLyuksyutovRegime}

\end{document}